\documentclass[12pt,a4paper]{article}
\usepackage{mathtext}
\usepackage[cp1251]{inputenc}
\usepackage[english]{babel}
\usepackage{euscript,amssymb,amsfonts,amsmath,amsthm}
\usepackage{mathrsfs}
\usepackage[T2A]{fontenc}
\usepackage{color}

\newtheorem{theorem}{\indent \sc Theorem}
\newtheorem{lemma}{\indent \sc Lemma}
\newtheorem{cor}{\indent \sc Corollary}

\theoremstyle{remark}
\newtheorem{myremark}{\indent \sc Remark}

\newcommand{\R}{\mathbf R}
\newcommand{\N}{\mathbf N}
\newcommand{\E}{{\sf E}}

\newcommand{\Prob}{{\sf P}}

\renewcommand{\le}{\leqslant}
\renewcommand{\ge}{\geqslant}

\newcommand{\la}{\lambda}

\title{On the accuracy of the approximation of the\\
complex exponent by the first terms of\\
its Taylor expansion with applications\thanks{Research supported by
the Russian Foundation for Basic Research (projects 11-01-00515a,
11-07-00112a, 11-01-12026-ofi-m) and by the grant of the President
of Russia (MK--2256.2012.1).}}
\author{Irina Shevtsova\thanks{Faculty of Computational Mathematics and
Cybernetics, Lomonosov Moscow State University, Leninskie Gory,
GSP-1, Moscow, 119991, Russia; Institute for Informatics Problems of
the Russian Academy of Sciences; e-mail: ishevtsova@cs.msu.su}}
\date{}

\begin{document}

\maketitle

\begin{abstract}

%{\bf Short abstract:}
%
%A new bound for the remainder term in the Taylor expansion of the
%complex exponent $e^{ix}$, $x\in\R$, is proved yielding precise
%moment-type estimates of the accuracy of the approximation of the
%characteristic function (the Fourier--Stieltjes transform) of a
%probability distribution by the first terms of its Taylor expansion.
%Moreover, a precise upper bound for the third moment of a
%probability distribution in terms of the absolute third moment is
%established which sharpens Jensen's inequality. Based on these
%results, new improved bounds for characteristic functions and their
%derivatives are obtained that are {\it uniform} in the class of
%distributions with fixed first three moments.
%
%{\bf Another abstract (long):}

A new bound for the remainder term in the Taylor expansion of the
complex exponent $e^{ix}$, $x\in\R$, is proved yielding precise
moment-type estimates of the accuracy of the approximation of the
characteristic function (the Fourier--Stieltjes transform) of a
probability distribution by the first terms of its Taylor expansion.
Namely, for an arbitrary random variable $X$ with the characteristic
function $f(t)=\E e^{itX}$, $t\in\R$, and $\E X=0$, $\E X^2=1$,
$\E|X|^3=b\geqslant1$, the symbol $\E$ standing for the mathematical
expectation, the precise bounds
$$
|\E X^3| \leqslant c(b)\E|X|^3,
$$
\begin{eqnarray*}
\left|f(t)-1+t^2/2\right|&\leqslant&
\inf_{\lambda\geqslant0}\left(\lambda|\E
X^3|+q_3(\lambda)\E|X|^3\right) {|t|^3}/{6} \leqslant b\gamma_3(b)
{|t|^3}/{6},
\\
\left|f'(t)+t\right|&\leqslant&
\inf_{\lambda\geqslant0}\left(\lambda|\E
X^3|+q_2(\lambda)\E|X|^3\right) {t^2}/{2} \leqslant
b\gamma_2(b){t^2}/{2},
\\
\left|f''(t)+1\right|&\leqslant&
\inf_{\lambda\geqslant0}\left(\lambda|\E
X^3|+q_1(\lambda)\E|X|^3\right)|t| \leqslant b\gamma_1(b)|t|
\end{eqnarray*}
are proved for all $t\in\R$ and $b\geqslant1$, where the function
$\displaystyle c(b)=\sqrt{0.5\sqrt{1+8b^{-2}}+0.5 -2b^{-2}}$
increases strictly monotonically varying within the limits
$0=c(1)\leqslant c(b)<\lim\limits_{b\to\infty}c(b)=1$,
$$
q_n(\lambda)= \sup_{x>0}\frac{n!}{x^n}\bigg|
e^{ix}-\sum_{k=0}^{n-1}\frac{(ix)^k}{k!}
-\lambda\frac{(ix)^n}{n!}\bigg|,\quad \gamma_n(b)=
\inf_{\lambda>0}(\lambda c(b)+q_n(\lambda)).
$$
Moreover, the functions $\gamma_n(b)$ increase strictly
monotonically varying within the limits
$\gamma_n(1)\leqslant\gamma_n(b)<\lim\limits_{b\to\infty}\gamma_n(b)=1$,
$n=1,2,3,$ with $\gamma_3(1)<0.5950,$ $\gamma_2(1)=2/\pi<0.6367,$
$\gamma_1(1)<0.7247$.
\end{abstract}

\smallskip


\smallskip

(secondary)

%60E10   Characteristic functions; other transforms
%
%60E15   Inequalities; stochastic orderings
%
%26A06   One-variable calculus
%
%26A09   Elementary functions
%
%41A10 Approximation by polynomials
%
%41A44 Best constants,
%
%41A80 Remainders in approximation formulas
%
%42A38   Fourier and Fourier-Stieltjes transforms and other
%transforms of Fourier type
%
%49K35   Minimax problems

\smallskip

\section{Introduction and motivation}

As is well known, the remainder term
$$
r_n(x) = e^{ix}-\sum_{k=0}^{n-1}\frac{(ix)^k}{k!},\quad x\in\R,\
n\in\N,
$$
in the Taylor expansion of the complex exponent satisfies the
precise inequality
\begin{equation}\label{ComplexExpTaylorApproxFromManuals}
|r_n(x)| \le{|x|^n}/{n!}, \quad x\in\R,\ n\in\N,
\end{equation}
with equality attained as $x\to0$, i. e. the factor $1/n!$ on the
r.-h. side of~\eqref{ComplexExpTaylorApproxFromManuals} cannot be
made less. Nevertheless, this does not mean that
inequality~\eqref{ComplexExpTaylorApproxFromManuals} is
unimprovable. Indeed, in 1991 H.\,Prawitz~\cite{Prawitz1991}
suggested to rearrange a part of the remainder (however, always a
smaller part) to the main term and proved that:
\begin{equation}\label{ComplexExpTaylorApproxPrawitz}
\bigg|r_n(x) -\frac{n}{2(n+1)}\cdot\frac{(ix)^n}{n!}\bigg|
\le\frac{n+2}{2(n+1)}\cdot\frac{|x|^n}{n!}, \quad x\in\R,\ n\in\N,
\end{equation}
with equality still attained as $x\to0$,
whence~\eqref{ComplexExpTaylorApproxFromManuals} immediately
follows.

The advantage of bound~\eqref{ComplexExpTaylorApproxPrawitz} as
compared with~\eqref{ComplexExpTaylorApproxFromManuals} becomes
especially noticeable, if $x$ is an integration variable. For
example, in probability theory $x$ may stand for the product $tX$ of
an arbitrary random variable (r.v.) $X$ defined on some probability
space $(\Omega,\mathcal A,{\sf P})$ and an argument $t\in\R$ of its
characteristic function (ch.f.)
$$
f(t)=\E e^{itX}=\int_{-\infty}^{+\infty} e^{itx}\, dF(x),\quad
F(x)={\sf P}(X<x),\ x\in\R,
$$
which is the Fourier--Stieltjes transform of the function of bounded
variation $F(x)$ (the distribution function of the r.v. $X$).
Namely, suppose that for some $n\in\N$\\[-2mm]
$$
\E|X|^n\equiv \int_{-\infty}^{+\infty}|x|^n dF(x)<\infty
$$
and denote
$$
\alpha_k=\E X^k,\quad \beta_k=\E|X|^k,\quad k=1,2,\ldots,n,
$$
$$
R_n(t)=\E\,r_n(tX)= \int_{-\infty}^{\infty}r_n(tx)\,dF(x)= f(t)
-\sum_{k=0}^{n-1}\frac{\alpha_k(it)^k}{k!}.
$$
Then, by virtue of the Jensen inequality, $|\alpha_k|\le\beta_k$,
$k=1,\ldots,n$. Moreover, $\alpha_n$ may vanish for odd $n$, for
example, for any symmetric distribution (i.e., if the r.v.'s $X$ and
$(-X)$ have identical distributions), whereas $\beta_n$ may be
infinitely large.

As it follows from~\eqref{ComplexExpTaylorApproxFromManuals},
\begin{equation}\label{ChFTaylorExpApproxFromManuals}
|R_n(t)|\leqslant \frac{\beta_n|t|^n}{n!},\quad t\in\R,
\end{equation}
with equality attained at any degenerate distribution as $t\to0$,
i.e. the factor $1/n!$ on the r.-h. side
of~\eqref{ChFTaylorExpApproxFromManuals} cannot be made less.
However, by use of inequality~\eqref{ComplexExpTaylorApproxPrawitz},
Prawitz managed to replace the absolute moment $\beta_n$ by the
linear combination of $|\alpha_n|$ and $\beta_n$ with coefficients
still summing up to one:
\begin{equation}\label{ChFTaylorExpApproxPrawitz}
|R_n(t)|\leqslant \frac{n|\alpha_n| +(n+2)\beta_n}{2(n+1)}\cdot
\frac{|t|^n}{n!},\quad t\in\R,
\end{equation}
whence~\eqref{ChFTaylorExpApproxFromManuals} immediately follows by
virtue of Jensen's inequality. Prawitz also paid a special attention
to the case $n=3$, which is very important in the problem of
estimation of the accuracy of the normal approximation to normalized
sums of independent random variables with finite third moments, and
in the same paper~\cite{Prawitz1991} noted that the coefficient
$$
\frac{n+2}{2(n+1)!}=\frac3{48}=0.1041\ldots
$$
at $\beta_3$ on the r.-h. side of~\eqref{ChFTaylorExpApproxPrawitz}
cannot be less than
$$
\varkappa_3\equiv\sup_{x>0}(\cos x-1+x^2/2)/x^3=0.0991\ldots\ .
$$

Inequality~\eqref{ComplexExpTaylorApproxPrawitz} stipulates natural
questions: if a larger part of the remainder is rearranged to the
main term, will the factor $(n+2)/(2(n+1)!)$ on the r.-h. side
of~\eqref{ComplexExpTaylorApproxPrawitz} become less or not? If yes,
then what is its least possible value and will the sum of the
coefficients at the corresponding main term and remainder still be
equal to one or will it increase? To answer these questions, we
propose to consider the functions
$$
q_n(\lambda)= \sup_{x>0}\frac{n!}{x^n}\bigg|
e^{ix}-\sum_{k=0}^{n-1}\frac{(ix)^k}{k!}
-\lambda\frac{(ix)^n}{n!}\bigg|,\quad \la\ge0,\ n=1,2,\ldots\
$$
(although here the supremum over $x>0$ can be replaced by the
supremum over all $x\in\R$, $x\neq0$, we will use a less cumbersome
variant), which guarantee the validity of the inequality
$$
\bigg|e^{ix}-\sum_{k=0}^{n-1}\frac{(ix)^k}{k!}
-\la\cdot\frac{(ix)^n}{n!}\bigg| \le q_n(\la)\cdot\frac{|x|^n}{n!},
\quad x\in\R,\ n\in\N,\ \la\ge0.
$$
Eliminating the real or the imaginary part in the definition of
$q_n(\lambda)$ we observe that
\begin{multline}\label{ComplexExpTaylorMyApprox_qLowBound}
\inf_{\la\ge0}q_n(\lambda)\ge
\\
\ge \left\{
 \begin{array}{l}
\displaystyle \sup_{x>0}\frac{n!}{x^n}\bigg|
\Re\bigg(e^{ix}-\sum_{k=0}^{n-1}\frac{(ix)^k}{k!}\bigg) \bigg|
=\sup_{x>0}\frac{n!}{x^n}\bigg|\cos x-
\sum_{k=0}^{(n-1)/2}\frac{(-1)^kx^{2k}}{(2k)!}\bigg|,\ \ n\hbox{ is
odd},
\\[2mm]
\displaystyle \sup_{x>0}\frac{n!}{x^n}\bigg|
\Im\bigg(e^{ix}-\sum_{k=0}^{n-1}\frac{(ix)^k}{k!}\bigg) \bigg|
=\sup_{x>0}\frac{n!}{x^n}\bigg|\sin x-
\sum_{k=1}^{n/2}\frac{(-1)^{k-1}x^{2k-1}}{(2k-1)!}\bigg|,\ \ n\hbox{
is even}.
  \end{array}
\right.
\end{multline}
Replacing the supremum $\sup_{x>0}$ in the definition of
$q_n(\lambda)$ by the limit $\lim_{x\to0+}$ we also notice that
$$
q_n(\lambda)\ge\lim_{x\to0}\frac{n!}{x^n}\bigg|
e^{ix}-\sum_{k=0}^{n-1}\frac{(ix)^k}{k!}
-\lambda\frac{(ix)^n}{n!}\bigg|\vee
\lim_{x\to\infty}\frac{n!}{x^n}\bigg|
e^{ix}-\sum_{k=0}^{n-1}\frac{(ix)^k}{k!}
-\lambda\frac{(ix)^n}{n!}\bigg|=|1-\la|\vee\la,
$$
for all $\la\ge0$, hence we will consider only the interval
$0\le\la<1/2$. Inequality~\eqref{ComplexExpTaylorApproxPrawitz}
implies that
$$
q_n(\la)=1-\la,\quad 0\le\la\le \frac{n}{2(n+1)}.
$$
Define
$$
\la_*=\la_*(n)=\sup\{\la\ge0\colon \sup_{0\le
s\le\la}(s+q_n(s))=1\},
$$
$$
\la^*=\la^*(n)=\inf\{\la\ge0\colon
q_n(\la)=\inf_{s\ge0}q_n(s)\},\quad n\in\N,
$$
i.\,e. $\la_*$ is the greatest value of $\la$ that minimizes the sum
$\la+q_n(\la)$, and $\la^*\ge\la_*$ is the least value of $\la$ that
minimizes $q_n(\la)$. Then, actually, only $\la\in[\la_*,\la^*]$ are
of interest. Inequality~\eqref{ComplexExpTaylorApproxPrawitz} also
implies that
$$
\la_*\ge\frac{n}{2(n+1)},
$$
and the posed questions can be re-formulated as follows:
$$
\la_*(n)\stackrel{?}{=} \frac{n}{2(n+1)},\quad
\la^*(n)\stackrel{?}{>}\la_*(n),\quad q_n(\la^*(n))=\,?
$$
Moreover, using the introduced functions it is easy to obtain the
following estimates for $R_n(t)$ and its derivatives, which
improve~\eqref{ChFTaylorExpApproxPrawitz}.

\begin{theorem}\label{ThChFTaylorExpApprox}
For any r.v. $X$ with the characteristic function $f(t)$ and
$\E|X|^n<\infty$ for some $n\in\N$, for all $t\in\R$ and
$\lambda\ge0$ the following estimates hold:
\begin{equation}\label{ChFTaylorExpApprox}
\bigg|f(t) -\sum_{k=0}^{n-1}\frac{\alpha_k(it)^k}{k!}
-\lambda\frac{\alpha_n(it)^n}{n!}\bigg|\le
q_n(\lambda)\frac{\beta_n|t|^n}{n!},
\end{equation}
\begin{equation}\label{ChFDerTaylorExpApprox}
\bigg|\frac{d^\ell f(t)}{dt^\ell}-\sum_{k=0}^{n-\ell-1}
i^{\ell+k}\alpha_{\ell+k}\frac{t^k}{k!} -\lambda
\frac{i^n\alpha_nt^{n-\ell}}{(n-\ell)!} \bigg|\leqslant
q_{n-\ell}(\la)\frac{\beta_n|t|^{n-\ell}}{(n-\ell)!},\quad
\ell=\overline{1,(n-1)}.
\end{equation}
\end{theorem}

\begin{myremark}
If $\la_*$, $\la^*$ are known, then, actually,
in~\eqref{ChFTaylorExpApprox}, \eqref{ChFDerTaylorExpApprox} the
greatest lower bounds over $\la\ge0$ can be replaced by those over
smaller sets $\la_*\le\la\le\la^*$, and it suffices to study the
properties of the functions $q_n(\la)$ only within the intervals
$\la\in[\la_*,\la^*]$.
\end{myremark}

In lemmas~\ref{LemComplexExpTaylorApprox_n=3},
\ref{LemComplexExpTaylorApprox_n=2},
\ref{LemComplexExpTaylorApprox_n=1} below, it will be demonstrated
that for $n=1,2,3$
$$
\la_*(1)=\frac14=0.25,\ \la_*(2)=\frac13=0.3333\ldots,\
\la_*(3)=\frac38=0.375,
$$
$$
\lambda^*(1) = \frac{\sin\theta_1^*}{\theta_1^*}=0.3108\ldots, \quad
\la^*(2)=4\pi^{-2}=0.4052\ldots,\quad
\la^*(3)=6\,\frac{\theta_3^*-\sin\theta_3^*}{(\theta_3^*)^3}
=0.4466\ldots,
$$
where $\theta_1^*=2.3311\ldots,$ $\theta_3^*=3.9958\ldots$ are,
respectively, the unique roots of the equations
$$
\theta_1\sin\theta_1 +\cos\theta_1-1=0,\quad \theta_1\in(0,\pi),
$$
$$
\theta_3^2+2\theta_3\sin\theta_3 +6(\cos\theta_3-1)=0,\quad
\theta_3\in(0,2\pi),
$$
i.\,e. the functions $\la+q_n(\la)$ are constant (and equal to one)
within the intervals $0\le\la\le n/(2(n+1))=\la_*(n)$, increase
strictly monotonically for $\la_*(n)\le\la\le\la^*(n)$, and the
functions $q_n(\la)$ decrease strictly monotonically for
$0\le\la\le\la^*(n)$ and attain their minimum values at
$\la=\la^*(n)$, $n=1,2,3$. In addition, in
lemmas~\ref{LemComplexExpTaylorApprox_n=3},
\ref{LemComplexExpTaylorApprox_n=2},
\ref{LemComplexExpTaylorApprox_n=1} below it will be proved that
\begin{eqnarray*}
\inf_{\la\ge0}q_1(\la)&=&q_1(\la^*(1))= \sup_{x>0}\frac{1-\cos
x}{x}= \frac{1-\cos\theta_1^*}{\theta_1^*}=0.7246\ldots,
\\
\inf_{\la\ge0}q_2(\la)&=&q_2(4\pi^{-2})= 2\sup_{x>0}\frac{x-\sin
x}{x^2}=2\,\frac{x-\sin x}{x^2}\Big|_{x=\pi} =\frac2\pi
=0.6366\ldots,
\\
\inf_{\la\ge0}q_3(\la)&=&q_3(\la^*(3))= 6\sup_{x>0}\frac{\cos
x-1+x^2/2}{x^3}=
6\frac{\cos\theta_3^*-1+(\theta_3^*)^2/2}{(\theta_3^*)^3}=
0.5949\ldots (=6\varkappa_3),
\end{eqnarray*}
i.\,e., actually, for $n=1,2,3$
inequalities~\eqref{ComplexExpTaylorMyApprox_qLowBound} hold with
the equality sign. In other words, say, for $n=3$ one can eliminate
the imaginary part of $(e^{ix}-1-ix-(ix)^2/2-\lambda(ix)^3/6)$ when
searching the supremum in the definition of $q_3(\lambda)$ by
choosing a special value of $\lambda=\lambda^*(3)$:
$$
\inf_{\lambda\ge0}\sup_{x>0}\frac1{x^3}\sqrt{\Big(\cos
x-1+\frac{x^2}2\Big)^2 +\Big(\sin x -x+\frac{\la x^3}6\Big)^2}=
\sup_{x>0}\frac{\cos x-1+{x^2}/2}{x^3}.
$$

Note that the estimates for $R_n(t)$ and its derivatives
\begin{eqnarray}
\label{ChFTaylorExpApprox_n=3}
\left|f(t)-1-i\alpha_1t+\alpha_2t^2/2\right|&\le&
\min_{3/8\le\la\le\la^*(3)}(\la|\alpha_3|+q_3(\la)\beta_3)
{|t|^3}/{6},
\\
\label{ChFDer1TaylorExpApprox_n=3}
\left|f'(t)-i\alpha_1+\alpha_2t\right|&\le&
\min_{1/3\le\la\le4\pi^{-2}}(\la|\alpha_3|+q_2(\la)\beta_3)
{t^2}/{2},
\\
\label{ChFDer2TaylorExpApprox_n=3} \left|f''(t)+\alpha_2\right|&\le&
\min_{1/4\le\la\le\la^*(1)}(\la|\alpha_3|+q_1(\la)\beta_3)|t|,
\\
\label{ChFTaylorExpApprox_n=1} \left|f(t)-1\right|&\le&
\min_{1/4\le\la\le\la^*(1)}(\la|\alpha_1|+q_1(\la)\beta_1) |t|,
\end{eqnarray}
implied by theorem~\ref{ThChFTaylorExpApprox} for $n=1,2,3$ are
precise in the sense that equalities in
\eqref{ChFTaylorExpApprox_n=3}--\eqref{ChFTaylorExpApprox_n=1} are
attained for each $|t|\le\theta$ at the symmetric three-point
distributions of the form
$\Prob(|X|=\theta/|t|)=t^2/\theta^2=1-\Prob(X=0)$ (for which
$f(t)=1-t^2\theta^{-2}(1-\cos\theta)$, $\alpha_1=\alpha_3=0$,
$\alpha_2=1$, $\beta_1=|t|/\theta$, $\beta_3=\theta/|t|$) with
$\theta=\theta_3^*$ in~\eqref{ChFTaylorExpApprox_n=3}, $\theta=\pi$
in~\eqref{ChFDer1TaylorExpApprox_n=3}, and $\theta=\theta_1^*$
in~\eqref{ChFDer2TaylorExpApprox_n=3},~\eqref{ChFTaylorExpApprox_n=1}.

The following theorem allows to get rid of the third moment
$\alpha_3$
in~\eqref{ChFTaylorExpApprox_n=3}--\eqref{ChFDer2TaylorExpApprox_n=3}
with $\E X$, $\E(X-\E X)^2$, $\E|X-\E X|^3$ being fixed.

\begin{theorem}\label{Th_|alpha3|<=beta3}
For all $b\ge1$ and any r.v. $X$ with $\E X=0$, $\E X^2=1$,
$\E|X|^3=b$
$$
\left|\E X^3\right|\le A(b)\E|X|^3,
$$
where
$$
A(b)=\sqrt{\frac12\sqrt{1+8b^{-2}}+\frac12-2b^{-2}}<1,
$$
with equality attained for each $b\ge1$ at the two-point
distribution
$$
\Prob\bigg(X=\Big({\frac{1\mp u} {1\pm u}}\Big)^{1/2} \bigg)
=\frac{1\pm u}2, \quad u=\sqrt{b\sqrt{b^2+8}/2-b^2/2-1}.
$$
Moreover, the function $A(b)$ is concave and increases strictly
monotonically varying within the limits $0=A(1)\le
A(b)<\lim\limits_{b\to\infty}A(b)=1$, $b\ge1$. The function $bA(b)$,
$b\ge1$, is concave as well.
\end{theorem}

Theorem~\ref{Th_|alpha3|<=beta3} improves Jensen's inequality, which
states that $|\alpha_3|/\beta_3\le1$: actually, this ratio is
strictly less than one for all distributions with zero mean and only
tends to one as the normalized third moment $\beta_3/\beta_2^{3/2}$
goes to infinity.

Theorems~\ref{ThChFTaylorExpApprox} and~\ref{Th_|alpha3|<=beta3}
imply

\begin{cor}\label{CorChFTaylorExpApproxNOalpha3}
For all $b\ge1$ and any r.v. $X$ with $\E X=0$, $\E X^2=1$,
$\E|X|^3=b$ the following inequalities hold for all $t\in\R$:
\begin{eqnarray*}
\left|f(t)-1+t^2/2\right|&\le& b\gamma_3(b) {|t|^3}/{6},
\\
\left|f'(t)+t\right|&\le& b\gamma_2(b){t^2}/{2},
\\
\left|f''(t)+1\right|&\le& b\gamma_1(b)|t|,
\\
\left|f(t)-1\right|&\le& (\theta_1^*)^{-1}(1-\cos\theta_1^*)\E|tX|
\le 0.7247\cdot|t|\E|X|,
\end{eqnarray*}
where
$$
\gamma_n(b)= \min_{\la_*(n)\le\la\le\la^*(n)}(\la
A(b)+q_n(\la)),\quad n=1,2,3,
$$
moreover, the functions $b\gamma_n(b)$, $\gamma_n(b)$, $n=1,2,3,$
are concave and increase strictly monotonically in $b\ge1$,
$\gamma_n(b)$ varying within the limits
$$
q_n(\la^*(n))=\gamma_n(1)\le
\gamma_n(b)<\lim\limits_{b\to\infty}\gamma_n(b)=1,\quad b\ge1,\quad
n=1,2,3.
$$
\end{cor}

The values of the functions $\gamma_n(b)$, $n=1,2,3,$ for some
$b\ge1$ are presented in columns 2, 5, 8 of
table~\ref{TabChFgamma_n(b)+lambda_n(b),n=123}. In columns 3, 6, 9,
the values of $\la_n=\la_n(b)$ are specified, that deliver minimum
in the definition of $\gamma_n(b)$, and in columns 4, 7, 10 the
values of $q_n=q_n(\la_n(b))$ are presented as well.

\begin{table}[h]
\centering\small
\begin{tabular}{||c||c|c|c||c|c|c||c|c|c||}
\hline $b=$&$\gamma_1(b)\le$&$\la_1\ge$&$q_1\le$
&$\gamma_2(b)\le$&$\la_2\ge$&$q_2\le$
&$\gamma_3(b)\le$&$\la_3\ge$&$q_3\le$\\
\hline
1 & 0.724612 & 0.3108 & 0.7247 & 0.636620 & 0.4052 & 0.6367 & 0.594972 & 0.4466 & 0.5950\\
1.0001 & 0.729674 & 0.3091 & 0.7247 & 0.643222 & 0.4033 & 0.6367 & 0.602250 & 0.4447 & 0.5950\\
1.001 & 0.740517 & 0.3057 & 0.7248 & 0.657374 & 0.3992 & 0.6368 & 0.617864 & 0.4407 & 0.5952\\
1.005 & 0.759711 & 0.2999 & 0.7253 & 0.682462 & 0.3924 & 0.6374 & 0.645582 & 0.4340 & 0.5957\\
1.01 & 0.773696 & 0.2960 & 0.7258 & 0.700771 & 0.3877 & 0.6380 & 0.665840 & 0.4293 & 0.5964\\
1.05 & 0.828077 & 0.2821 & 0.7293 & 0.772182 & 0.3714 & 0.6422 & 0.745088 & 0.4130 & 0.6005\\
1.10 & 0.863075 & 0.2743 & 0.7325 & 0.818315 & 0.3621 & 0.6460 & 0.796466 & 0.4038 & 0.6043\\
1.20 & 0.903490 & 0.2662 & 0.7370 & 0.871750 & 0.3526 & 0.6512 & 0.856138 & 0.3943 & 0.6095\\
1.30 & 0.927590 & 0.2618 & 0.7399 & 0.903693 & 0.3473 & 0.6547 & 0.891887 & 0.3890 & 0.6130\\
1.40 & 0.943762 & 0.2590 & 0.7421 & 0.925160 & 0.3440 & 0.6573 & 0.915944 & 0.3857 & 0.6156\\
1.50 & 0.955288 & 0.2570 & 0.7436 & 0.940474 & 0.3417 & 0.6591 & 0.933121 & 0.3834 & 0.6174\\
1.60 & 0.963824 & 0.2556 & 0.7448 & 0.951825 & 0.3400 & 0.6605 & 0.945860 & 0.3817 & 0.6188\\
1.70 & 0.970322 & 0.2546 & 0.7457 & 0.960468 & 0.3387 & 0.6616 & 0.955565 & 0.3804 & 0.6199\\
1.79 & 0.975371 & 0.2538 & 0.7464 & 0.967189 & 0.3378 & 0.6624 & 0.963114 & 0.3795 & 0.6208\\
1.90 & 0.979362 & 0.2531 & 0.7470 & 0.972502 & 0.3370 & 0.6631 & 0.969083 & 0.3787 & 0.6214\\
2.00 & 0.982560 & 0.2526 & 0.7475 & 0.976760 & 0.3364 & 0.6637 & 0.973868 & 0.3781 & 0.6220\\
3.00 & 0.995576 & 0.2506 & 0.7494 & 0.994102 & 0.3341 & 0.6659 & 0.993365 & 0.3757 & 0.6243\\
4.00 & 0.998416 & 0.2502 & 0.7498 & 0.997888 & 0.3336 & 0.6664 & 0.997624 & 0.3752 & 0.6248\\
5.00 & 0.999306 & 0.2501 & 0.7499 & 0.999075 & 0.3334 & 0.6666 & 0.998959 & 0.3751 & 0.6249\\
$\infty$ & $1$ & $1/4$ & $3/4$ & $1$ & $1/3$ & $2/3$ & $1$ & $3/8$ & $5/8$\\
\hline
\end{tabular}
\caption{The values of the functions $\gamma_n(b)$ rounded up for
some $b\ge1$ (columns 2, 5, 8), the corresponding values of
$\la_n=\la_n(b)$ rounded down, that deliver minimum in the
definition of $\gamma_n(b)$ (columns 3, 6, 9), and the values of
$q_n=q_n(\la_n(b))$ rounded up (columns 4, 7, 10) for $n=1,2,3$.}
\label{TabChFgamma_n(b)+lambda_n(b),n=123}
\end{table}

The problem of estimation of the accuracy of the approximation of
characteristic functions by polynomials was also considered
in~\cite{Rozovsky1985}.

Note that the estimates given in
corollary~\ref{CorChFTaylorExpApproxNOalpha3} are rather rough
either for large~$t$, or for large~$b$. However, this defect can be
corrected if the characteristic function $f(t)$ is approximated by
its derivatives (and the derivatives~--- by the characteristic
function). Namely, the following estimates can be derived from
corollary~\ref{CorChFTaylorExpApproxNOalpha3} and the results
of~\cite{KorolevShevtsova2010SAJ, Shevtsova2012ShapeBias} which are
obtained with the application of the zero biased and shape biased
transformations:
\begin{eqnarray*}
\Big|f(t)+\frac{f'(t)}t\Big|&\le&
2\sin\Big(\frac{bt}4\wedge\frac\pi2\Big)\wedge
\Big(\gamma_2(b)\cdot\frac{b|t|}{2}+\frac{t^2}2 \Big),
\\
\left|f(t)+f''(t)\right|&\le&
2\sin\Big(\frac{b|t|}2\wedge\frac\pi2\Big) \wedge\Big(
\gamma_1(b)\cdot b|t|+\frac{t^2}2 \Big)
\end{eqnarray*}
for all $t\in\R$ and any r.v. $X$ with $\E X=0$, $\E X^2=1$,
$\E|X|^3=b\ge1$. Note that the r.h.-sides of the last inequalities
remain bounded for large $t$ as well as for large $b$.

The presented estimates for characteristic functions allow to
sharpen substantially the Berry--Esseen inequality and its
structural improvements (see, e.g., the recent works
\cite{Shevtsova2006,Tyurin2009DAN,Tyurin2010TVP,
KorolevShevtsova2010DAN1,KorolevShevtsova2010DAN2,KorolevShevtsova2009,
Tyurin2009arxiv,Tyurin2010UMN,KorolevShevtsova2010SAJ,
Shevtsova2010DAN_0.56,KorolevShevtsova2010TVP,
Tyurin2011,Shevtsova2011arxiv} and references
in~\cite{KorolevShevtsova2009,KorolevShevtsova2010SAJ}), non-uniform
estimates of the accuracy of the normal approximation to
distributions of sums of independent r.v.'s
(see~\cite{Paditz1977,Michel1981,Tysiak1983,Paditz1986,
NefedovaShevtsova2012,NefedovaShevtsova2011DAN} and references
therein), as well as uniform and non-uniform moment-type estimates
of the rate of convergence in limit theorems for compound and mixed
compound Poisson distributions (see
\cite{Michel1993,KorolevShorgin1997,KorolevShevtsova2010DAN2,
KorolevShevtsova2010SAJ,NefedovaShevtsova2011DAN}).

\section{Proofs}

The following lemmas establish the properties of the functions
$q_n(\la)$ and give the exact values of the quantities $\la_*(n)$,
$\la^*(n)$, for $n=3$ (lemma~\ref{LemComplexExpTaylorApprox_n=3}),
$n=2$ (lemma~\ref{LemComplexExpTaylorApprox_n=2}), and $n=1$
(lemma~\ref{LemComplexExpTaylorApprox_n=1}).

\begin{lemma}[see \cite{Prawitz1973}]\label{LemPrawitsCos3}
Let $\theta_3^*=3.9958\ldots$ be the unique root of the equation
$$
x^2+2x\sin x +6(\cos x-1)=0,\quad x\in(0,2\pi).
$$
Then
$$
\varkappa_3\equiv \sup_{x>0}\frac{\cos x-1+x^2/2}{x^3}=
\frac{\cos\theta_3^*-1+(\theta_3^*)^2/2}{(\theta_3^*)^3}=
\frac{\theta_3^*-\sin\theta_3^*}{3(\theta_3^*)^2}=0.099161\ldots
$$
\end{lemma}

\begin{lemma}\label{LemComplexExpTaylorApprox_n=3}
Let
$$
\lambda^* =
6\frac{\theta_3^*-\sin\theta_3^*}{(\theta_3^*)^3}=0.4466\ldots\ .
$$
For $3/8<\la\le\la^*$ by $\theta_3(\la)\in(0,2\pi)$ denote the
unique root of the equation
$$
2\cos x(\la x^4-18x^2+36)-6x\sin x(x^2(\la+1)-12) -(3-4\la)x^4-72=0,
\quad x\in(0,2\pi),
$$
and $\theta_3(\la)=0$ for $0\le\la\le3/8$. Then
$\theta_3(\la^*)=\theta_3^*$, $q_3(\la)=1-\la$ for $0\le\la\le3/8$
and
$$
q_3(\la)=\frac6{x^3}\sqrt{\Big(\cos x-1+\frac{x^2}2\Big)^2
+\Big(\sin x -x+\frac{\la
x^3}6\Big)^2}\,\bigg|_{x=\theta_3(\la)},\quad \frac38<\la\le\la^*.
$$
Moreover, the function $\la+q_3(\la)$ is strictly increasing for
$3/8<\la\le\la^*$, the function $q_3(\la)$ is strictly decreasing
for $0\le\la\le\la^*$. In particular,
$$
q_3(3/8)=5/8,\quad \inf_{\la\ge0}q_3(\la)=
q_3(\la^*)=6\cdot\varkappa_3 =0.594971\ldots\ .
$$
\end{lemma}

\begin{proof} Denote
$$
h(x)=h(x,\la)= \Big(\cos x-1+\frac{x^2}2\Big)^2 +\Big(\sin x
-x+\frac{\la x^3}6\Big)^2,\quad x>0,
$$
$$
f(x)=f(x,\la)=6x^{-3}\sqrt{h(x,\la)},\quad x>0.
$$
Then $h(x,\la)>0$, $f(x,\la)>0$ for all $x>0$, and
$$
q_3(\la)=\sup\limits_{x>0}f(x,\la),\quad \la\ge0.
$$
From a result of~\cite{Prawitz1991} it follows that
$q_3(3/8)\le5/8$, hence, for all $0\le\la\le3/8$
$$
q_3(\la)\le q_3(3/8) +3/8-\la\le 5/8+3/8-\la=1-\la.
$$
Since $q_3(\la)\ge\lim\limits_{x\to0+}f(x,\la)=|1-\la|,$ $\la\ge0$,
we conclude that $q_3(\la)=1-\la$, $\theta_3(\la)=0$ for all
$0\le\la\le3/8$.

Now consider $3/8<\la<1/2$. For all $x>0$ we have
$$
f'(x)=3\frac{xh'(x)-6h(x)}{x^4\sqrt{h(x)}},
$$
\begin{eqnarray*}
h'(x)&=&2(\cos x-1+x^2/2)(x-\sin x) +2(\sin x-x+\la x^3/6)(\cos x-1
+kx^2/2),
\\
f_1(x)&\equiv& f'(x)\cdot2x^4 \sqrt{h(x)} =6(xh'(x)-6h(x))=
\\
& =& 2\cos x(\la x^4-18x^2+36) -6x\sin x(x^2(\la+1)-12) +4\la
x^4-3x^4-72,
\\
f_2(x)&\equiv& \frac{f_1'(x)}{2x^2} = (\la-3)x\cos x - \sin x (\la
x^2+9\la-9) +8\la x-6x,
\\
f_2'(x)&=& 3(1-\la)x\sin x + (6-8\la-\la x^2)\cos x +8\la-6,
\\
f_2''(x)&=& \sin x(\la x^2+5\la-3) + (3-5\la)x\cos x,
\\
f_2'''(x)&=& x(7\la\sin x -3\sin x +\la x \cos x) = x\cos x\cdot
g(x),\quad g(x)=\la x + (7\la-3)\tan x.
\end{eqnarray*}
Evidently, $f_1(\theta_3(\la))\equiv0$ by the definition of
$\theta_3(\la)$. Split the domain $x>0$ into the non-overlapping
intervals $x\in(2\pi m,2\pi(m+1)]\equiv(0,2\pi]+2\pi m$,
$m=0,1,2,\ldots,$ where
$$
A+c=\{x\in\R\colon x=a+c,\ a\in A\},\quad A\subset\R,\ c\in\R,
$$
and consider the function $f(x)$ and its derivatives on each of
these intervals. The function $\cos x$ has the zeros $\pi/2+2\pi m$
and $3\pi/2+2\pi m$ in the interval $x\in(0,2\pi]+2\pi m$, which
might be the zeros of the function $f_2'''(x)$ as well. However,
$$
f_2'''(\pi/2+2\pi m)=(\pi/2+2\pi m)\cdot(7\la-3)\neq0,\quad
f_2'''(3\pi/2+2\pi m)=(3\pi/2+2\pi m)\cdot(3-7\la)\neq0
$$
for all $\la\neq3/7$, thus all the roots of the equation
$f_2'''(x)=0$ coincide with those of the function $g(x)$, if
$\la\neq3/7$. Now consider three cases for possible values of
$\la\in(3/8,1/2)$:

\begin{enumerate}

\item if $3/8<\la<3/7$, then the function $g(x)$ vanishes in
the points $x_1\in(0,\pi/2)+2\pi m$ and $x_2\in(\pi,3\pi/2)+2\pi m$
changing its sign from $+$ to $-$. Since $\cos x_1>0$, $\cos x_2<0$,
the function $f_2'''(x)$ changes its sign on each of the intervals
$x\in(0,2\pi]+2\pi m$, $m\ge0,$ only in the two points
$x_1\in(0,\pi/2)$ (from $+$ to $-$) and $x_2\in(\pi,3\pi/2)$ (from
$-$ to $+$).

\item if $\la=3/7$, then $f_2'''(x)=\la x^2\cos x$ changes its sign  in
the two points: $x_1=\pi/2+2\pi m$ (from $+$ to $-$) and
$x_2=3\pi/2+2\pi m$ (from $-$ to $+$).

\item if $3/7<\la<1/2$, then the function $g(x)$ vanishes in
the points $x_1\in(\pi/2,\pi)+2\pi m$ and $x_2\in(3\pi/2,2\pi)+2\pi
m$ changing its sign from $-$ to $+$. Since $\cos x_1<0$, $\cos
x_2>0$, the function $f_2'''(x)$ changes its sign on each of the
intervals $x\in(0,2\pi]+2\pi m$, $m\ge0,$ only in the two points
$x_1\in(\pi/2,\pi)+2\pi m$ (from $+$ to $-$) and
$x_2\in(3\pi/2,2\pi)+2\pi m$ (from $-$ to $+$).

\end{enumerate}

Summarizing what was said above we conclude that on each of the
intervals $x\in(0,2\pi]+2\pi m$, $m\ge0,$ the function $f_2'''(x)$
changes its sign exactly in two points $x_1\in(0,\pi)+2\pi m$ (from
$+$ to $-$) and $x_2\in(\pi,2\pi)+2\pi m$ (from $-$ to $+$). Thus,
$x_1$ is the point of maximum and $x_2$ is the point of minimum of
the function
$$
f_2''(x)=\sin x(\la x^2+5\la-3) + (3-5\la)x\cos x.
$$
We have
$$
f_2''(0)=0,\quad f_2''(2\pi m)=2\pi m(3-5\la)>0,\ m\ge1,
$$
$$
f_2''(\pi+2\pi m)= -(3-5\la)(\pi+2\pi m)<0,
$$
hence, $f_2''(x_1)>0$, $f_2''(x_2)<0$, and $f_2''(x)$ changes its
sign exactly in two points $x_3\in(x_1,\pi)+2\pi m
\subset(0,\pi)+2\pi m$ (from $+$ to~$-$) and $x_4\in(\pi,2\pi)+2\pi
m$ (from $-$ to~$+$).

Thus, the function
$$
f_2'(x)= 3(1-\la)x\sin x + (6-8\la-\la x^2)\cos x +8\la-6
$$
has exactly two stationary points on each of the intervals
$(0,2\pi]+2\pi m$, $m\ge0$: $x_3\in(0,\pi)+2\pi m$ which is the
point of maximum and $x_4\in(\pi,2\pi)+2\pi m$ which is the point of
minimum. For $m=0$ we have $f_2'(0)=0$, $f_2'(2\pi)=-\la(2\pi)^2<0$,
consequently, $f_2'(x_3)>0$, $f_2'(x_4)<0$, and the function
$f_2'(x)$ changes its sign in a unique point $x_6\in(0,2\pi)$ (from
$+$ to $-$). For $m\ge1$ we have
$$
f_2'(2\pi m)=-\la (2\pi m)^2<0,
$$
$$
f_2'(\pi+2\pi m) = \la(16+(\pi+2\pi m)^2)-12\ge
3/8\cdot(16+9\pi^2)-12>0,
$$
consequently, $f_2'(x_3)>0$, $f_2'(x_4)<0$, and $f_2'(x)$ changes
its sign exactly in two points $y_5\in(0,\pi)+2\pi m$ (from $-$ to
$+$) and $y_6\in(\pi,2\pi)+2\pi m$ (from $+$ to $-$).

Thus, the function
$$
f_2(x)\equiv \frac{f_1'(x)}{2x^2} = (\la-3)x\cos x - \sin x (\la
x^2+9\la-9) +8\la x-6x,
$$
has a unique stationary point $x_6$ on the interval $(0,2\pi]$ which
is the point of maximum and exactly two stationary points on each of
the intervals $(0,2\pi]+2\pi m$ with $m\ge1$: $y_5\in(0,\pi)+2\pi m$
which is the point of minimum and $y_6\in(\pi,2\pi)+2\pi m$ which is
the point of maximum. Since
$$
f_2(0)=0,\quad f_2(2\pi m) = 18\pi m(\la-1) <0,\ m\ge1,
$$
we conclude that $f_2(x_6)>0$ and the function $f_1'(x)=2x^2f_2(x)$
changes its sign on the interval $(0,2\pi]$ in a unique point
$x_8\in(x_6,2\pi)\in(0,2\pi)$ (from $+$ to $-$). With $m\ge1$ we
have
\begin{multline*}
f_2(3\pi/2+2\pi m) =\la(x^2+8x+9)-6x-9\big|_{x=3\pi/2+2\pi m}\ge
\\
\ge 3/8\cdot(x^2+8x+9)-6x-9\big|_{x=3\pi/2+2\pi m}
=3/8\cdot((3\pi/2+2\pi m-4)^2-31)\ge
\\
\ge  3/8\cdot((7\pi/2-4)^2-31)>0,
\end{multline*}
and hence, the function $f_1'(x)=2x^2f_2(x)$ changes its sign
exactly in two points on each of the intervals $(0,2\pi]+2\pi m$,
$m\ge1$: $y_7\in(y_5,3\pi/2)+2\pi m\in(0,3\pi/2)+2\pi m$ (from $-$
to $+$) and $y_8\in(3\pi/2,2\pi)+2\pi m$ (from $+$ to $-$).

Thus, the function
$$
f_1(x)=f'(x)\cdot2x^4\sqrt{h(x)} =2\cos x(\la x^4-18x^2+36) -6x\sin
x(x^2(\la+1)-12) +(4\la -3)x^4-72,
$$
where $h(x)>0$, has a unique stationary point $x_8$ on the interval
$(0,2\pi]$ (the point of maximum), and exactly two stationary points
on each of the intervals $(0,2\pi]+2\pi m$ with $m\ge1$:
$y_7\in(0,3\pi/2)+2\pi m$ (the point of minimum) and
$y_8\in(3\pi/2,2\pi)+2\pi m$ (the point of maximum). Since
$$
f_1(0)=0,\quad f_1(2\pi m)=3(2\pi m)^2((2\la-1)(2\pi
m)^2-12)<0,\quad m\ge1,
$$
the function $f'(x)$ has a unique zero within the interval
$(0,2\pi)$, which is the point of maximum of the function $f(x)$ and
coincides with $\theta_3(\la)$.

As regards the domain $x>2\pi$, we are going to prove that
$f_1(x)<0$ for all $x>2\pi$ and $3/8<\la\le\la^*$, implying that the
function $f(x)$ has no maxima for $x>2\pi$ and completing the proof
of the relation $q_3(\la)=6f(\theta_3(\la),\la)$. Since the function
$f_1(x)$ has a unique point of maximum $y_8\in(3\pi/2,2\pi)+2\pi m$
on each of the intervals $(0,2\pi]+2\pi m$, $m\ge1,$ it suffices to
prove that $f_1(x)<0$ for all $x\in(3\pi/2,2\pi)+2\pi m$.

Note that $\cos x>0$, $\sin x<0$ for $x\in(3\pi/2,2\pi)+2\pi
m=(0,\pi/2]+3\pi/2+2\pi m$, and hence for $\la\le\la^*<0.4467$
$$
f_1(x)\le 2\cos x(\la^* x^4-18x^2+36) -6x\sin x(x^2(\la^*+1)-12)
+(4\la^* -3)x^4-72,
$$
moreover, as it can be easily seen, $\la^* x^4-18x^2+36>0$,
$x^2(\la^*+1)-12>0$ for all
$$
x\ge 3\pi/2+2\pi m\ge7\pi/2=10.99\ldots\ .
$$

Split the domain $x\in(0,\pi/2]+3\pi/2+2\pi m$ into two intervals:
$x\in(0,\pi/4)+3\pi/2+2\pi m$ and $x\in[\pi/4,\pi/2]+3\pi/2+2\pi m$
and examine the function $f_1(x)$ on each of them. For
$x\in(0,\pi/4)+3\pi/2+2\pi m$ we have $\cos x\le\sqrt2/2$, $\sin
x\ge-1$ and thus
\begin{multline*}
f_1(x)\le \sqrt2(\la^* x^4-18x^2+36) +6x(x^2(\la^*+1)-12) +(4\la^*
-3)x^4-72=
\\
=(\la^*(\sqrt2+4)-3)x^4+6(\la^*+1)x^3-18\sqrt2x^2-72x+36(\sqrt2-2)<
\\
< -x^2\big((3-\la^*(\sqrt2+4))x^2-6(\la^*+1)x+18\sqrt2\big).
\end{multline*}
Since $3-\la^*(\sqrt2+4)>1-1/\sqrt2>0$, now it can be easily seen
that $f_1(x)<0$ for all
$$
x> \frac{3(\la^*+1)
+\sqrt{9(\la^*+1)^2-18\sqrt2(3-\la^*(\sqrt2+4))})}
{3-\la^*(\sqrt2+4)}=10.91\ldots,
$$
in particular, for $x\ge 3\pi/2+2\pi m\ge7\pi/2=10.99\ldots\ .$

For $x\in[\pi/4,\pi/2]+3\pi/2+2\pi m$ we have $\cos x\le1$, $\sin
x\ge-\sqrt2/2$ and thus
\begin{multline*}
f_1(x)\le 2(\la^* x^4-18x^2+36) +3\sqrt2x(x^2(\la^*+1)-12) +(4\la^*
-3)x^4-72=
\\
=3(2\la^*-1)x^4+3\sqrt2(\la^*+1)x^3-36x^2-36\sqrt2x<
\\
< -3x^2\big((1-2\la^*)x^2-\sqrt2(\la^*+1)x+12\big)<0
\end{multline*}
for all $x\in\R$, since the discriminant
$2(\la^*+1)^2-4\cdot12(1-2\la^*)<-0.93$ is negative.

Thus, we have proved that the function $f(x)=f(x,\la)$ attains its
maximal value for $x>0$ at the unique point
$x=\theta_3(\la)\in(0,2\pi)$ for $3/8<\la\le\la^*$ and at the point
$x\to0+$ for $0\le\la\le3/8$.

Now prove that $q_3(\la^*)=6\varkappa_3$. With
$$
\la=\la^*=6\frac{\theta_3^*-\sin\theta_3^*}{(\theta_3^*)^3}
=18\frac{\varkappa_3}{\theta_3^*} = 18\frac{\cos
\theta_3^*-1+(\theta_3^*)^2/2}{(\theta_3^*)^4}
$$
(two last relations following from lemma~\ref{LemPrawitsCos3}), we
have
\begin{eqnarray*}
h(\theta_3^*,\la^*)&=&\varkappa_3^2(\theta_3^*)^6,
\\
h'_x(x,\la^*)|_{x=\theta_3^*}
&=&2(\cos\theta_3^*-1+(\theta_3^*)^2/2) (\theta_3^*-\sin\theta_3^*)+
\\
&&+2(\sin\theta_3^*-\theta_3^*+\la^*(\theta_3^*)^3/6)
(\cos\theta_3^*-1 +\la^*(\theta_3^*)^2/2)=
\\
&=& 2\varkappa_3(\theta_3^*)^3\cdot 3\varkappa_3(\theta_3^*)^2=
6\varkappa_3^2(\theta_3^*)^5,
\\
f_1(\theta_3^*)/6&=&
xh'_x(x,\la)|_{x=\theta_3^*}-6h(\theta_3^*,\la)=0.
\end{eqnarray*}
By virtue of the uniqueness of the root $\theta_3(\la)$ of the
equation $f_1(x)=0$, which is equivalent to $f'_x(x,\la)=0$ within
the interval $(0,2\pi)$, we conclude that
$\theta_3(\la^*)=\theta_3^*$ so that
$$
q_3(\la^*)=f(\theta_3(\la^*),\la^*)=f(\theta_3^*,\la^*)=
6\frac{\cos\theta_3^*-1+(\theta_3^*)^2/2}{(\theta_3^*)^3}
=6\varkappa_3.
$$

Now prove the declared properties of the functions $q_3(\la)$,
$\la+q_3(\la)$. Since
$$
f''_{\la\la}(x,\la)=\frac{x^3(\cos x-1+x^2/2)^2}{6 \big((\cos
x-1+{x^2}/2)^2 +(\sin x -x+{\la x^3}/6)^2\big)^{3/2}}
>0,\quad 0<x<\infty,
$$
the function $f(x,\la)$ is strictly convex in $\lambda$ for all
$x\in(0,\infty)$. As it follows from what was proved, the least
upper bound in the definition of $q_3(\la)$ is attained for all
$3/8<\la\le\la^*$ at a finite point $x=\theta_3(\la)$ separated from
zero:
$$
q_3(\la)=\sup_{x>0}f(x,\la)=\max_{\theta_3(\la)\le x\le2\pi}
f(x,\la),\quad \theta_3(\la)>0,\quad 3/8<\la\le\la^*,
$$
hence for all $\la_1,\la_2\in(3/8,\la^*]$ and $0\le\alpha\le1$ we
have
\begin{multline*}
q_3(\alpha\la_1+(1-\alpha)\la_2) =\max_{\theta_3(\la)\le x\le2\pi}
f(x,\alpha\la_1+(1-\alpha)\la_2)<
\\
< \max_{\theta_3(\la)\le x\le2\pi} \big(\alpha f(x,\la_1)
+(1-\alpha)f(x,\la_2)\big)\le
\\
\le\alpha\sup_{x>0}f(x,\la_1) +(1-\alpha)\sup_{x>0}f(x,\la_2)=
\alpha q_3(\la_1) +(1-\alpha)q_3(\la_2),
\end{multline*}
i.\,e. the function $q_3(\la)$ is strictly convex for
$3/8<\la\le\la^*$ as well. Since
$$
q_3(\la)\ge 6\sup_{x>0}(\cos x-1+x^2/2)/x^3
=6\varkappa_3=q_3(\la^*),
$$
$\la=\la^*$ being the unique point of minimum  of the function
$q_3(\la)$ on the interval $3/8<\la\le\la^*$, the function
$q_3(\la)$ should decrease strictly monotonically for
$3/8<\la\le\la^*$. For $0\le\la\le3/8$, obviously, the function
$q_3(\la)=1-\la$ is strictly decreasing.

The function $\la+q_3(\la)$ is strictly convex for $3/8<\la\le\la^*$
as a sum of a convex and a strictly convex functions, hence, it
cannot be constant on any subinterval of the interval
$3/8<\la\le\la^*$. On the other hand, $\la+q_3(\la)\ge1$ for all
$0\le\la\le\la^*$, thus, $\la+q_3(\la)$ should strictly increase for
$3/8<\la\le\la^*$.
\end{proof}

%%%%%%%%%%%%%%%%%%%%%%%%%%%%%%%%%%%%%%%%%%%%%%%%%%%%%%%%%%%%%%
%                       q_2(lambda)
%%%%%%%%%%%%%%%%%%%%%%%%%%%%%%%%%%%%%%%%%%%%%%%%%%%%%%%%%%%%%%

%\begin{lemma}\label{Lem_sup(x-sin(x))/x^2=1/pi}
%For all $x>0$
%$$
%\frac{x-\sin x}{x^2}\le \frac1\pi=0.3183\ldots,
%$$
%with equality attained at the unique point $x=\pi$.
%\end{lemma}
%
%\begin{proof} Consider the function $f(x)=(x-\sin x)/x^2$, $x>0$. We
%have
%$$
%f'(x)\cdot x^3=4\sin x-2x(1+\cos x)=8\cos\frac x2\Big(\sin\frac x2
%-\frac x2\cos\frac x2\Big).
%$$
%Since $\sin y-y\cos y>0$ for all $0<y<\pi$, all the stationary
%points of the function $f(x)$ in the interval $(0,2\pi)$ coincide
%with zeros of the function $\cos\frac x2$ in the interval
%$(0,2\pi)$, i.\,e. $f(x)$ has the unique point of maximum $x=\pi$ in
%the interval $(0,2\pi)$ and $f(\pi)=\pi^{-1}$. For $x\ge2\pi$ we
%have
%$$
%\frac{x-\sin x}{x^2}\le \frac1x+\frac1{x^2}\le \frac1{2\pi}
%+\frac1{4\pi^2}< \frac1\pi,
%$$
%which completes the proof of the lemma.
%\end{proof}

\begin{lemma}\label{LemComplexExpTaylorApprox_n=2}
For $1/3<\la\le4\pi^{-2}=0.4052\ldots$ let $\theta_2(\la)\in(0,\pi]$
be the unique root of the equation
$$
x(8-\la x^2)\sin x +4(\la x^2+x^2-4)\sin^2\frac x2 - 4x^2=0, \quad
0<x\le\pi,
$$
and let $\theta_2(\la)=0$ for $0\le\la\le1/3$. Then
$\theta_2(4\pi^{-2})=\pi$, $q_2(\la)=1-\la$ for $0\le\la\le1/3$ and
$$
q_2(\la)=2\sqrt{\Big(\frac{\cos x-1+{\la x^2}/2}{x^2}\Big)^2
+\Big(\frac{x-\sin x}{x^2}\Big)^2}\,\bigg|_{x=\theta_2(\la)}, \quad
\frac13<\la\le\frac{4}{\pi^{2}}.
$$
Moreover, the function $\la+q_2(\la)$ is strictly increasing for
$1/3<\la\le4\pi^{-2}$, the function $q_2(\la)$ is strictly
decreasing for $0\le\la\le4\pi^{-2}$. In particular,
$$
q_2(1/3)=2/3,\quad \inf\limits_{\la\ge0}q_2(\la)
=q_2(4\pi^{-2})=2\sup_{x>0}\frac{x-\sin x}{x^2}=\frac2\pi
=0.636619\ldots\ .
$$
\end{lemma}

\begin{proof}
Denote
$$
f(x)=f(x,\la)=2\sqrt{\Big(\frac{\cos x-1+{\la x^2}/2}{x^2} \Big)^2
+\Big(\frac{x-\sin x}{x^2}\Big)^2},\quad x>0.
$$
Then $f(x,\la)>0$, $x>0$, and
$$
q_2(\la)=\sup\limits_{x>0}f(x,\la),\quad \la\ge0.
$$
From the result of~\cite{Prawitz1991} it follows that
$q_2(1/3)\le2/3$, hence, for all $0\le\la\le1/3$
$$
q_2(\la)\le q_2(1/3) +1/3-\la\le 2/3+1/3-\la=1-\la.
$$
Since $q_2(\la)\ge\lim\limits_{x\to0+}f(x,\la)=|1-\la|,$ $\la\ge0$,
we conclude that $q_2(\la)=1-\la$ with $\theta_2(\la)=0$ for all
$0\le\la\le1/3$.

Now assume that $1/3<\la\le4\pi^{-2}$. Consider two cases of
possible values of $x$:

\begin{enumerate}

\item
$0<x<2\pi$. We have
\begin{eqnarray*}
f_1(x)&\equiv& f'(x)\cdot x^5f(x)/2= x(8-\la x^2)\sin x +4(\la
x^2+x^2-4)\sin^2(x/2) - 4x^2,
\\
f_2(x)&\equiv& {f_1'(x)}/{x} = (4-4\la-\la x^2)\cos x +(2-\la)x\sin
x +4(\la-1),
\\
f_2'(x)&=& (2-3\la)x\cos x + (\la x^2+3\la-2)\sin x,
\\
f_2''(x)&=& \la x^2\cos x+(5\la-2)x\sin x = x\cos x\cdot g(x),\quad
g(x)=\la x + (5\la-2)\tan x.
\end{eqnarray*}

Obviously, $f_1(\theta_2(\la))\equiv0$ by the definition of
$\theta_2(\la)$. The function $\cos x$ has the zeros $\pi/2$ and
$3\pi/2$ within the interval $x\in(0,2\pi)$, which might be the
zeros of the function $f_2''(x)$ as well. However,
$$
f_2''(\pi/2)=(5\la-2)\cdot\pi/2\neq0,\quad
f_2''(3\pi/2)=(2-5\la)\cdot3\pi/2\neq0
$$
for all $\la\neq2/5$, thus all the roots of the equation
$f_2''(x)=0$ coincide with those of the function $g(x)$, if
$\la\neq2/5$. Now consider three cases of possible values of
$\la\in(1/3,4\pi^{-2}]$:

\begin{enumerate}

\item
if $1/3<\la<2/5$, then the function $g(x)$ vanishes in some points
$x_1\in(0,\pi/2)$ and $x_2\in(\pi,3\pi/2)$ where it changes its sign
from $+$ to $-$. Since $\cos x_1>0$, $\cos x_2<0$, the function
$f_2''(x)$ changes its sign only in two points $x_1\in(0,\pi/2)$
(from $+$ to $-$) and $x_2\in(\pi,3\pi/2)$ (from $-$ to $+$).

\item
if $\la=2/5$, then $f_2''(x)=\la x^2\cos x$ changes its sign in two
points $x_1=\pi/2$ (from $+$ to $-$) and $x_2=3\pi/2$ (from $-$ to
$+$).

\item
if $2/5<\la\le4\pi^{-2}$, then the function $g(x)$ vanishes in some
points $x_1\in(\pi/2,\pi)$, $x_2\in(3\pi/2,\pi,2\pi)$ where it
changes its sign from $-$ to $+$. Since $\cos x_1<0$, $\cos x_2>0$,
the function $f_2''(x)$ changes its sign only in two points
$x_1\in(\pi/2,\pi)$ (from $+$ to $-$) and $x_2\in(3\pi/2,2\pi)$
(from $-$ to $+$).

\end{enumerate}

Summarizing what was said above we conclude that the function
$f_2''(x)$ changes its sign on the interval $(0,2\pi)$ exactly in
two points $x_1\in(0,\pi)$ (from $+$ to $-$) and $x_2\in(\pi,2\pi)$
(from $-$ to $+$). Thus, $x_1\in(0,\pi)$ is the point of maximum and
$x_2\in(\pi,2\pi)$ is the point of minimum of the function
$$
f_2'(x)= (2-3\la)x\cos x + (\la x^2+3\la-2)\sin x.
$$
We have
$$
f_2'(0)=0,\quad f_2'(\pi)=\pi(3\la-2)<0,\quad
f_2'(2\pi)=2\pi(2-3\la)>0,
$$
hence, $f_2'(x)$ changes its sign exactly in two points
$x_3\in(0,\pi)$ (from $+$ to~$-$) and $x_4\in(\pi,2\pi)$ (from $-$
to~$+$).

Thus, the function
$$
f_2(x)\equiv{f_1'(x)}/{x} = (4-4\la-\la x^2)\cos x +(2-\la)x\sin x
+4(\la-1)
$$
has exactly two stationary points $x_3\in(0,\pi)$ which is the point
of maximum and $x_4\in(\pi,2\pi)$ which is the point of minimum. We
have $f_2(0)=0$, $f_2(2\pi)=-\la(2\pi)^2<0$, hence, the function
$f_2(x)$ changes its sign in a unique point $x_5\in(0,2\pi)$ (from
$+$ to $-$), which is the unique point of maximum of the function
$$
f_1(x)\equiv f'(x)\cdot x^5f(x)/2= x(8-\la x^2)\sin x +4(\la
x^2+x^2-4)\sin^2(x/2) - 4x^2.
$$
Since $f_1(0)=0$, $f_1(\pi)=4(\la\pi^2-4)\le0$ for all
$\la\le4\pi^{-2}$, we conclude that the function $f_1(x)$ changes
its sign in a unique point $x_6\in(0,\pi]$ (from $+$ to $-$), which
is the unique point of maximum of $f(x)$ and coincides with
$\theta_2(\la)$, since $f'(x,\la)|_{x=\theta_2(\la)}\equiv0$.

\item
$x\ge2\pi$. For $\la\le4\pi^{-2}$ we have
$$
f(x,\la)=2\sqrt{\Big(\frac{\cos x-1}{x^2}+\frac\la2 \Big)^2
+\Big(\frac{x-\sin x}{x^2}\Big)^2}\le
2\sqrt{\Big(\frac{2}{x^2}+\frac2{\pi^2} \Big)^2
+\Big(\frac{1+x}{x^2}\Big)^2}\le
$$
$$
\le 2\sqrt{\Big(\frac{1}{2\pi^2}+\frac2{\pi^2} \Big)^2
+\Big(\frac{1}{4\pi^2}+\frac1{2\pi}\Big)^2}<0.6268<\frac2\pi
=0.6366\ldots\ .
$$
\end{enumerate}

Summarizing what was said above we conclude that the function
$f(x)=f(x,\la)$ attains its maximum value for $x>0$ at the unique
point $x=\theta_2(\la)\in(0,2\pi)$, if $1/3<\la\le4\pi^{-2}$, and at
the point $x\to0+$, if $0\le\la\le1/3$.

Prove that $q_2(4\pi^{-2})=2\pi^{-1}$. Since
$f'_x(x,4\pi^{-2})|_{x=\pi}=0$, we conclude that
$\theta_2(4\pi^{-2})=\pi$ by virtue of the uniqueness of the root
$\theta_2(\la)$ of the equation $f'_x(x,\la)=0$. Hence,
$$
q_2(4\pi^{-2})=f(\pi,4\pi^{-2})=2\sqrt{\Big(\frac{\cos
x-1+2\pi^{-2}x^2}{x^2}\Big)^2 +\Big(\frac{x-\sin
x}{x^2}\Big)^2}\,\bigg|_{x=\pi} =\frac2{\pi}.
$$

Now prove the properties of the functions $q_2(\la)$,
$\la+q_2(\la)$. Since
$$
f''_{\la\la}(x,\la)=\frac{x^2(x-\sin x)^2}{2 \big((\cos
x-1+\la{x^2}/2)^2 +(x-\sin x)^2\big)^{3/2}}
>0,\quad 0<x<\infty,
$$
the function $f(x,\la)$ is strictly convex in $\lambda\ge0$ for all
$x\in(0,\infty)$. As it follows from what has already been proved,
the least upper bound in the definition of $q_2(\la)$ is attained
for all $1/3<\la\le4\pi^{-2}$ at the finite point $x=\theta_2(\la)$
separated from zero:
$$
q_2(\la)=\sup_{x>0}f(x,\la)=\max_{\theta_2(\la)\le x\le2\pi}
f(x,\la),\quad \theta_2(\la)>0,\quad 1/3<\la\le4\pi^{-2},
$$
hence for all $\la_1,\la_2\in(1/3,4\pi^{-2}]$ and $0\le\alpha\le1$
we have
\begin{multline*}
q_2(\alpha\la_1+(1-\alpha)\la_2) =\max_{\theta_2(\la)\le x\le2\pi}
f(x,\alpha\la_1+(1-\alpha)\la_2)<
\\
< \max_{\theta_2(\la)\le x\le2\pi} \big(\alpha f(x,\la_1)
+(1-\alpha)f(x,\la_2)\big)\le
\\
\le\alpha\sup_{x>0}f(x,\la_1) +(1-\alpha)\sup_{x>0}f(x,\la_2)=
\alpha q_2(\la_1) +(1-\alpha)q_2(\la_2),
\end{multline*}
i.\,e. the function $q_2(\la)$ is strictly convex for
$1/3<\la\le4\pi^{-2}$ as well. Since
$$
q_2(\la)\ge 2\sup_{x>0}(x-\sin x)/x^2 =2\pi^{-1}=q_2(4\pi^{-2}),
$$
$\la=4\pi^{-2}$ being the unique point of minimum  of the function
$q_2(\la)$ in the interval $1/3<\la\le4\pi^{-2}$, the function
$q_2(\la)$ should decrease strictly monotonically for
$1/3<\la\le4\pi^{-2}$. For $0\le\la\le1/3$, obviously, the function
$q_2(\la)=1-\la$ is strictly decreasing.

The function $\la+q_2(\la)$ is strictly convex for
$1/3<\la\le4\pi^{-2}$ as a sum of a convex and a strictly convex
functions, hence, it cannot be constant on any subinterval of the
interval $1/3<\la\le4\pi^{-2}$. On the other hand,
$\la+q_2(\la)\ge1$ for all $0\le\la\le4\pi^{-2}$, thus,
$\la+q_2(\la)$ should be strictly increasing for
$1/3<\la\le4\pi^{-2}$.
\end{proof}

%%%%%%%%%%%%%%%%%%%%%%%%%%%%%%%%%%%%%%%%%%%%%%%%%%%%%%%%%%%%%%
%                       q_1(lambda)
%%%%%%%%%%%%%%%%%%%%%%%%%%%%%%%%%%%%%%%%%%%%%%%%%%%%%%%%%%%%%%

\begin{lemma}\label{Lem_sup(1-cos(x))/x}
Let $\theta_1^*=2.3311\ldots$ be the unique root of the equation
$x\sin x +\cos x-1=0$ within the interval $(0,\pi)$. Then
$$
\varkappa_1\equiv \sup_{x>0}\frac{1-\cos x}{x}=
\frac{1-\cos\theta_1^*}{\theta_1^*}= \sin\theta_1^*=0.724611\ldots\
.
$$
\end{lemma}

\begin{proof}
Consider the function $f(x)=(1-\cos x)/x$, $x>0$. Since for
$x\ge2\pi$ we have
$$
f(x)\le\frac2{x}\le\frac1\pi<0.3184<\varkappa_1,
$$
it suffices only to consider $0<x<2\pi$. We have
$$
f'(x)\cdot x^2=x\sin x +\cos x-1=(x-\tan(x/2))\sin x.
$$
Since $f'(\pi)=-2\pi^{-2}\neq0$, all the zeros of $f'(x)$ within the
interval $(0,2\pi)$ coincide with those of the function
$g(x)=x-\tan(x/2)$. It is easy to see that within the interval
$(0,2\pi)$ the function $g(x)$ vanishes in a unique point
$x_1\in(0,\pi)$ changing its sign from $+$ to $-$. Since $\sin
x_1>0$, the function $f'(x)$ has a unique zero within the interval
$(0,2\pi)$, which coincides with $\theta_1^*$ and delivers maximum
to the function $f(x)$.
\end{proof}

\begin{lemma}\label{LemComplexExpTaylorApprox_n=1}
Let
$$
\lambda^* = \frac{\sin\theta_1^*}{\theta_1^*}=0.3108\ldots\ .
$$
For $1/4<\la\le\la^*$ let $\theta_1(\la)\in(0,\pi)$ be the unique
root of the equation
$$
\cos x(2-\la x^2)+(1+\la)x\sin x -2=0, \quad x\in(0,\pi),
$$
and $\theta_1(\la)=0$ for $0\le\la\le1/4$. Then
$\theta_1(\la^*)=\theta_1^*$, $q_1(\la)=1-\la$ for $0\le\la\le1/4$
and
$$
q_1(\la)=\sqrt{\Big(\frac{\cos x-1}{x}\Big)^2 +\Big(\frac{\sin x
-\la x}{x}\Big)^2} \,\bigg|_{x=\theta_1(\la)},\quad
\frac14<\la\le\la^*.
$$
Moreover, the function $\la+q_1(\la)$ is strictly increasing for
$1/4<\la\le\la^*$, the function $q_1(\la)$ is strictly decreasing
for $0\le\la\le\la^*$. In particular,
$$
q_1(1/4)=3/4,\quad \inf\limits_{\la\ge0}q_1(\la)=
q_1(\la^*)=\sup_{x>0}\frac{1-\cos
x}{x}\equiv\varkappa_1=0.724611\ldots\ .
$$
\end{lemma}

\begin{proof} Denote
$$
f(x)=f(x,\la)=x^{-1}\sqrt{(\cos x-1)^2 +(\sin x -\la x)^2},\quad
x>0.
$$
Then $f(x,\la)>0$, $x>0$, and
$$
q_1(\la)=\sup\limits_{x>0}f(x,\la),\quad \la\ge0.
$$
Notice that $q_1(\la)\ge\sup\limits_{x>0}(\cos
x-1)/x\equiv\varkappa_1>0.7246$ for all $\la\ge0$. From a result
of~\cite{Prawitz1991} it follows that $q_1(1/4)\le3/4$, hence, for
all $0\le\la\le1/4$
$$
q_1(\la)\le q_1(1/4) +1/4-\la\le 3/4+1/4-\la=1-\la.
$$
Since $q_1(\la)\ge\lim\limits_{x\to0+}f(x,\la)=|1-\la|,$ $\la\ge0$,
we conclude that $q_1(\la)=1-\la$, $\theta_1(\la)=0$ for all
$0\le\la\le1/4$.

Now assume that $1/4<\la<1/3$, in particular, $1/4<\la\le\la^*$.
Consider two cases of possible values of $x$:

\begin{enumerate}

\item
$0<x<2\pi$. We have
\begin{eqnarray*}
f_1(x)&\equiv& f'(x)\cdot x^2 \sqrt{(\cos x-1)^2 +(\sin x -\la x)^2}
= \cos x(2-\la x^2)+(1+\la)x\sin x -2,
\\
f_1'(x)&=& (1-\la)x\cos x +  (\la x^2+\la-1)\sin x,
\\
f_1''(x)&=& \la x^2\cos x + (3\la-1)x\sin x = x\cos \cdot g(x),\quad
g(x)=\la x + (3\la-1)\tan x.
\end{eqnarray*}

Obviously, $f_1(\theta_1(\la))\equiv0$ by the definition of
$\theta_1(\la)$. Within the interval $x\in(0,2\pi)$ the function
$x\cos x$ has the zeros $\pi/2$ and $3\pi/2$, which might be the
zeros of the function $f_1''(x)$ as well. However,
$$
f_1''(\pi/2)=(3\la-1)\cdot\pi/2>0,\quad
f_1''(3\pi/2)=(1-3\la)\cdot3\pi/2<0
$$
for all $\la<1/3$. Hence, all the roots of the equation $f_1''(x)=0$
coincide with the zeros of the function $g(x)$, for all
$0\le\la\le\la^*$. The function $g(x)$ vanishes in the points
$x_1\in(0,\pi/2)$ and $x_2\in(\pi,3\pi/2)$ changing its sign from
$+$ to $-$. Since $\cos x_1>0$, $\cos x_2<0$, the function
$f_1''(x)$ changes its sign only in two points $x_1\in(0,\pi/2)$
(from $+$ to $-$) and $x_2\in(\pi,3\pi/2)$ (from $-$ to $+$). Thus,
$x_1\in(0,\pi/2)$ is the point of maximum and $x_2\in(\pi,3\pi/2)$
is the point of minimum of the function $f_1'(x)$. We have
$$
f_1'(0)=0,\quad f_1'(\pi)=\pi(\la-1)<0,\quad
f_1'(2\pi)=2\pi(1-\la)>0,
$$
hence, $f_1'(x)$ changes its sign exactly in two points
$x_3\in(0,\pi)$ (from $+$ to~$-$) and $x_4\in(\pi,2\pi)$ (from $-$
to~$+$).

Thus, the function $f_1(x)$ has exactly two stationary points
$x_3\in(0,\pi)$, the point of maximum, and $x_4\in(\pi,2\pi)$, the
point of minimum. Since $f_1(0)=0$, $f_1(2\pi)=-\la(2\pi)^2<0$, the
function $f_1(x)$ changes its sign in a unique point
$x_5\in(0,2\pi)$ (from $+$ to $-$). Moreover,
$f_1(\pi)=\la\pi^2-4<\pi^2/3-4<0$, hence, $x_5\in(0,\pi)$, $x_5$
delivers maximum to $f(x)$ within the interval $(0,2\pi)$ and
coincides with $\theta_1(\la)$.

\item
$x\ge2\pi$. For $1/4<\la<1/3$ we have
$$
f(x,\la)=x^{-1}\sqrt{2(1-\cos x)-2\la x\sin x +\la^2 x^2}\le
x^{-1}\sqrt{4+2\la x +\la^2 x^2}\le
$$
$$
\le \sqrt{\pi^{-2}+(3\pi)^{-1}+3^{-2}}<0.5644 <\varkappa_1\le
\inf_{\la\ge0}q_1(\la).
$$
\end{enumerate}

Summarizing what was said above we conclude that the function
$f(x)=f(x,\la)$ attains its maximum value for $x>0$ at the unique
point $x=\theta_1(\la)\in(0,\pi)$, if $1/4<\la<1/3$, in particular,
if $1/4<\la\le\la^*$, and at the point $x\to0+$, if $0\le\la\le1/4$.

Prove that $q_1(\la^*)=\varkappa_1$. With
$$
\la=\la^*=\frac{\sin\theta_1^*}{\theta_1^*}=\frac{1-\cos
\theta_1^*}{(\theta_1^*)^2}
$$
(the last relation following from the definition of $\theta_1^*$
given in lemma~\ref{Lem_sup(1-cos(x))/x}), we have
\begin{multline*}
f_1(\theta_1^*)= \cos\theta_1^*(2-\la^*(\theta_1^*)^2)
+(1+\la^*)\theta_1^*\sin\theta_1^* -2=
\\
= \cos\theta_1^*(1+\cos\theta_1^*) +\theta_1^*\sin\theta_1^*
+\sin^2\theta_1^* -2= \cos\theta_1^* +\theta_1^*\sin\theta_1^* -1=0,
\end{multline*}
by the definition of $\theta_1^*$. By virtue of the uniqueness of
the root $\theta_1(\la)$ of the equation $f_1(x)=0$, which is
equivalent to $f'_x(x,\la)=0$ within the interval $(0,2\pi)$, we
conclude that $\theta_1(\la^*)=\theta_1^*$ and thus
$$
q_1(\la^*)=f(\theta_1(\la^*),\la^*)=f(\theta_1^*,\la^*)=
\frac{1-\cos\theta_1^*}{\theta_1^*}=\varkappa_1.
$$

Now prove the properties of the functions $q_1(\la)$,
$\la+q_1(\la)$. Since
$$
f''_{\la\la}(x,\la)=\frac{x(\cos x-1)^2}{\big((\cos x-1)^2 +(\sin
x-\la x)^2\big)^{3/2}}
>0,\quad 0<x<\infty,
$$
the function $f(x,\la)$ is strictly convex in $\lambda$ for all
$x\in(0,\infty)$. As it follows from what has been already proved,
the least upper bound in the definition of $q_1(\la)$ is attained
for all $1/4<\la<1/3$ at the finite point $x=\theta_1(\la)$
separated from zero:
$$
q_1(\la)=\sup_{x>0}f(x,\la)=\max_{\theta_1(\la)\le x\le2\pi}
f(x,\la),\quad \theta_1(\la)>0,\quad 1/4<\la<1/3,
$$
hence for all $\la_1,\la_2\in(1/4,1/3)$ and $0\le\alpha\le1$ we have
\begin{multline*}
q_1(\alpha\la_1+(1-\alpha)\la_2) =\max_{\theta_1(\la)\le x\le2\pi}
f(x,\alpha\la_1+(1-\alpha)\la_2)<
\\
< \max_{\theta_1(\la)\le x\le2\pi} \big(\alpha f(x,\la_1)
+(1-\alpha)f(x,\la_2)\big)\le
\\
\le\alpha\sup_{x>0}f(x,\la_1) +(1-\alpha)\sup_{x>0}f(x,\la_2)=
\alpha q_1(\la_1) +(1-\alpha)q_1(\la_2),
\end{multline*}
i.\,e. the function $q_1(\la)$ is strictly convex for $1/4<\la<1/3$
as well. Since
$$
q_1(\la)\ge \sup_{x>0}(1-\cos x)/x\equiv\varkappa_1=q_1(\la^*),
$$
$\la=\la^*$ being the unique point of minimum  of the function
$q_1(\la)$ on the interval $1/4<\la<1/3$, and thus, the function
$q_1(\la)$ should decrease strictly monotonically for
$1/4<\la\le\la^*$. For $0\le\la\le1/4$, the function
$q_1(\la)=1-\la$ is obviously strictly decreasing.

The function $\la+q_1(\la)$ is strictly convex for $1/4<\la<1/3$ as
a sum of a convex and a strictly convex functions, hence, it cannot
be constant on any subinterval of the interval $1/4<\la<1/3$. On the
other hand, $\la+q_1(\la)\ge1$ for all $0\le\la<1/3$, thus,
$\la+q_1(\la)$ should be strictly increasing for $1/4<\la<1/3$, in
particular, for $1/4<\la\le\la^*$.
\end{proof}

\begin{proof}[Proof of theorem~$\ref{Th_|alpha3|<=beta3}$]
From the results of~\cite{Hoeffding1955,MulhollandRogers1958} it
follows that the extremal value of the linear with respect to the
distribution function $F(x)=\Prob(X<x)$, $x\in\R$, functional
$$
\E X^3=\int_{-\infty}^\infty x^3\,dF(x)
$$
under the three linear moment-type conditions $\E X=0$, $\E X^2=1$,
$\E|X|^3=b$ is attained at a distribution concentrated in at most
four points (i.\,e. the distribution function $F(x)$ being constant
almost everywhere and having at most four jumps). For each $b\ge1$
there exists a unique two-point distribution which satisfies the
conditions $\E X=0$, $\E X^2=1$, $\E|X|^3=b$. This distribution is
given in the formulation of the theorem and turns the stated
inequality into equality. Thus, it remains to consider three- and
four-point distributions only.

Let $X$ take exactly three different values $x,y,z$ with the
corresponding probabilities ${p,q,r>0}$, $p+q+r=1$. Without loss of
generality it can be assumed that $x>y\ge0>z$. From the conditions
$\E X=0$, $\E X^2=1$ we find that
$$
p = \frac{1+yz}{(x-y)(x-z)},\quad q=-\frac{1+xz}{(x-y)(y-z)}, \quad
r=\frac{1+xy}{(x-z)(y-z)},\quad xz<-1<yz.
$$
Then
$$
\E X^3=x+y+z+xyz\equiv\alpha_3(x,y,z),
$$
$$
\E|X|^3=\frac{-z^3(1+xy)-z^2xy(x+y)-z(xy(1-xy)+x^2+y^2)+xy(x+y)}
{(y-z)(x-z)}\equiv \beta_3(x,y,z).
$$
The Lagrange function of the optimization problem
$\alpha_3(x,y,z)\to\sup$ under the constraint $\beta_3(x,y,z)=b$ has
the form
$$
f(x,y,z,\la)=\alpha_3(x,y,z) +\la(\beta_3(x,y,z)-b).
$$
In the stationary points we necessarily have
\begin{eqnarray*}
\frac{\partial}{\partial x}f(x,y,z,\la)&=&
(1+yz)\Big(1+\la\Big(1+\frac{2z^3}{(x-z)^2(y-z)}\Big)\Big)=0,
\\
\frac{\partial}{\partial y}f(x,y,z,\la)&=&
(1+xz)\Big(1+\la\Big(1+\frac{2z^3}{(x-z)(y-z)^2}\Big)\Big)=0.
\end{eqnarray*}
Since $xz<-1<yz$, from these equations we find that
$$
\frac{\la z^3(y-x)}{(x-z)^2(y-z)^2}=0,
$$
whence it follows that $\la=0$ by virtue of the conditions
$x>y\ge0>z$. If $\la=0$, then the condition $f'_x(x,y,z,\la)=0$
implies that $yz=-1$, i.\,e. $p=0$, that contradicts the condition
$xz<-1<yz$ and reduces the problem to checking two-point
distributions considered above.

Now let $X$ take exactly four values $t>u>v>w$ with the
corresponding probabilities $p,q,r,s>0$, $p+q+r+s=1$. From the
conditions $\E X=0$, $\E X^2=1$ we find that
$$
p=\frac{1+uv-s(u-w)(v-w)}{(t-u)(t-v)},\quad
q=-\frac{1+tv-s(t-w)(v-w)}{(t-u)(u-v)},
$$
$$
r=\frac{1+tu-s(t-w)(u-w)}{(t-v)(u-v)}.
$$
Then
$$
\alpha_3(s,t,u,v,w)\equiv\E X^3= t+u+v+tuv-s(t-w)(u-w)(v-w).
$$
Denote $\beta_3(s,t,u,v,w)=\E|X|^3$. Then the Lagrange function of
the optimization problem $\alpha_3(s,t,u,v,w)\to\sup$ under the
constraint $\beta_3(s,t,u,v,w)=b$ has the form
$$
f(s,t,u,v,w,\la)=\alpha_3(s,t,u,v,w) +\la(\beta_3(s,t,u,v,w)-b).
$$
For the proof of the theorem it suffices to consider two cases:

1) $t>u>v\ge0>w$. In this case
\begin{multline*}
\beta_3(s,t,u,v,w)\equiv\E|X|^3= \alpha_3(s,t,u,v,w)-2sw^3=
\\
=t+u+v+tuv-s(w^3+w^2(t+u+v)-w(tu+tv+uv)+tuv).
\end{multline*}
In the stationary points we necessarily have
\begin{eqnarray*}
f'_s(s,t,u,v,w,\la)&=& (1-\la)w^3
+(1+\la)\left(-w^2(t+u+v)+w(tu+tv+uv)-tuv\right)=0,
\\
f'_t(s,t,u,v,w,\la)&=& (1+\la)(1+uv-s(u-w)(v-w))\equiv
p(t-u)(t-v)(1+\la)=0,
\end{eqnarray*}
Since $p>0$ and $t>u>v$, the second equation implies that $\la=-1$.
With this value of $\la$ the first equation implies that $w=0$
contradicting the condition $w<0$. Thus, there are no extremal
distributions in this case.

2) $t>u\ge0\ge v>w$, $u\neq v$. In this case
$$
\beta_3(s,t,u,v,w)\equiv\E|X|^3= \alpha_3(s,t,u,v,w)-2rv^3-2sw^3.
$$
In the stationary points we necessarily have
\begin{eqnarray*}
f'_t(s,t,u,v,w,\la)&=& \frac{p(t-u)(2\la v^3
+(\la+1)(t-v)^2(u-v))}{(t-v)(u-v)}=0,
\\
f'_u(s,t,u,v,w,\la)&=& -\frac{q(t-u)(2\la v^3
+(\la+1)(t-v)(u-v)^2)}{(t-v)(u-v)}=0.
\end{eqnarray*}
With the account of the conditions $p,q>0$, $t>u>v$ these equations
imply $\la=-1$, $v=0$. With these values of $\la$ and $v$ we have
$$
f(s,t,u,0,w,-1)=\alpha_3(s,t,u,0,w)-\beta_3(s,t,u,0,w)+b= 2sw^3+b.
$$
In the stationary points we necessarily have
$$
f'_w(s,t,u,v,w,\la)=6sw^2=0,
$$
whence it follows that $w=0=v$ contradicting the condition $w<v$.
Thus, there are no extremal distributions in this case as well.

The properties of the function
$A(b)=\sqrt{\frac12\sqrt{1+8b^{-2}}+\frac12-2b^{-2}}$ can be
established by examination the derivatives. It is easy to see that
$$
A'(b)\cdot b^{3/2}A(b)=1-\big(1+8b^{-2}\big)^{-1/2}>0,\quad 1\le
b<\infty,
$$
i.\,e. $A(b)$ increases strictly monotonically for all $b\ge1$, and
$$
A''(b)\cdot b^8A^3(b)\big(1+8b^{-2}\big)^{3/2}/4
=16\big(1+8b^{-2}\big)^{1/2} -b\big(\big(b^2+8\big)^{1/2}+9b\big)-48
$$
decreases monotonically and, hence, attains its maximum value
$(-12)$ at the point $b=1$. Thus, $A''(b)<0$ for all $b\ge1$, i.\,e.
$A(b)$ is concave. For the function $bA(b)$ we have
$$
(bA(b))'' =\left(\Big({\frac
b2\sqrt{b^2+8}+\frac{b^2}2-2}\Big)^{1/2}\right)''
=-12\frac{b+\sqrt{b^2+8}}{b^3A^3(b)(b^2+8)^{3/2}}<0,\quad b>1,
$$
hence, $bA(b)$ is concave as well.
\end{proof}

%\bibliographystyle{plain}
%\bibliography{../bib/biblio,../bib/my_pub}

\end{document}